\documentclass[11pt]{amsart}

\usepackage[T1]{fontenc}
\usepackage[utf8]{inputenc}
\usepackage{lmodern}
\usepackage{microtype}
\usepackage{amsmath,amssymb,amsthm}
\usepackage{tikz}
\usepackage{pgfplots}
\usetikzlibrary{arrows.meta}
\pgfplotsset{compat=1.18}
\usepackage[colorlinks=true,linkcolor=blue,citecolor=blue,urlcolor=blue]{hyperref}
\usepackage[dvipsnames,svgnames]{xcolor}
\usepackage{mathtools}

\usepackage[foot]{amsaddr}

\usepackage[size=footnotesize]{todonotes}
\usepackage[capitalise,noabbrev]{cleveref}
\usepackage{autonum}

\definecolor{kupkablue}{RGB}{38,99,151}
\definecolor{kupkaorange}{RGB}{211,113,41}
\definecolor{kupkagray}{RGB}{92,101,110}
\tikzset{every picture/.style={line cap=round,line join=round}}

\newtheorem{theorem}{Theorem}[section]
\newtheorem{proposition}[theorem]{Proposition}
\newtheorem{lemma}[theorem]{Lemma}
\newtheorem{corollary}[theorem]{Corollary}
\numberwithin{equation}{section}

\theoremstyle{remark}
\newtheorem{remark}[theorem]{Remark}

\newcommand{\R}{\mathbb R}
\newcommand{\N}{\mathbb N}
\newcommand{\dd}{\,\mathrm d}
\newcommand{\Span}{\operatorname{span}}
\newcommand{\Abn}{\operatorname{Abn}}
\newcommand{\End}{\operatorname{End}}

\title[A sub-Riemannian endpoint map with no regular values]
{A sub-Riemannian endpoint map with no regular values}
\author{Antonio Lerario$^1$}
\author{Luca Rizzi$^1$}
\address{$^1$SISSA, via Bonomea 265, 34136 Trieste, Italy}
\email{lerario@sissa.it}
\email{lrizzi@sissa.it}
\author{Daniele Tiberio$^2$}
\address{$^2$Dipartimento di Matematica ``Tullio Levi-Civita'',
Universit\`a degli Studi di Padova, via Trieste 63,
35121 Padova, Italy}
\email{daniele.tiberio@unipd.it}
\date{\today}

\begin{document}

\begin{abstract}
We construct an endpoint map with no regular values, disproving the sub-Riemannian Sard conjecture. Remarkably, in our construction the minimizing Sard property remains valid.
 
\end{abstract}

\maketitle

\section{Introduction}

Let $M$ be a smooth, connected manifold, let $\Delta\subset TM$ be a smooth,
constant-rank, bracket-generating distribution, and fix a point
$q_0\in M$.  Let $\Omega_{q_0}$ denote the Hilbert manifold of
$W^{1,2}$ horizontal curves starting from $q_0$,
\begin{equation}
 \Omega_{q_0}\coloneq 
 \left\{\gamma\in W^{1,2}([0,1],M):
 \gamma(0)=q_0,\ \dot\gamma(t)\in\Delta_{\gamma(t)}
 \text{ for a.e. }t\right\}.
\end{equation}
Here $W^{1,2}$ may be defined using any auxiliary Riemannian metric on
$M$.  The endpoint map is the surjective, smooth map
\begin{equation}
 \mathcal E_{q_0}:\Omega_{q_0}\longrightarrow M,
 \qquad
 \mathcal E_{q_0}(\gamma)=\gamma(1).
\end{equation}
A horizontal curve $\gamma\in\Omega_{q_0}$ is \emph{singular} if
$D_\gamma\mathcal E_{q_0}$ is not surjective, and is then called
\emph{abnormal}.  We write
\begin{equation}
 \Abn_\Delta(q_0)
 \coloneq \{\mathcal E_{q_0}(\gamma):\gamma\in\Omega_{q_0}
 \text{ is singular}\}.
\end{equation}
When a global frame of $\Delta$ is available, the endpoint map can equivalently be
written in terms of the so-called \emph{control} of horizontal curves, and viewed as a function on $L^2$. This is the case for the
structure constructed below.

The Sard conjecture for endpoint maps, attributed to Zhitomirskii and
Montgomery, postulates that $\Abn_\Delta(q_0)$ has measure zero for every
smooth bracket-generating distribution and every base point; see
\cite[Section~10.2]{Montgomery}.  It was also explicitly formulated in
\cite[Section~4.3]{RiffordTrelat}.  Agrachev isolated a much weaker
question in \cite[Problem~III]{AgrachevOpen}: must an endpoint map have
at least one regular value?  Equivalently, can the singular curves
issuing from one point fill the whole manifold?  Recall that $q\in M$
is a regular value of $\mathcal E_{q_0}$ when every curve in
$\mathcal E_{q_0}^{-1}(q)$ is a regular point.  Thus
$\Abn_\Delta(q_0)=M$ means that the endpoint map has no regular values.

Our main result gives failure of
both properties, disproving the Sard conjecture of Zhitomirskii and
Montgomery and answering negatively Agrachev's regular-value question.

\begin{theorem}\label{thm:main}
There exist a smooth bracket-generating distribution $\Delta$ on
$\R^4$, of rank two and
step five, and a point $q_0\in\R^4$, such that:
\begin{enumerate}
 \item The endpoint map $\mathcal E_{q_0}$ has no regular
 values, equivalently
 \begin{equation}
  \Abn_\Delta(q_0)=\R^4.
 \end{equation}
 \item There exists a compact set $ \mathcal K\subset\Omega_{q_0}$, consisting entirely of singular curves, such that
 $\mathcal E_{q_0}(\mathcal K)$ contains a nonempty open subset of
 $\R^4$.  In particular, for any sub-Riemannian metric $g$ on $\Delta$, a nonempty open subset of
 $\Abn_\Delta(q_0)$ is reached by singular curves of uniformly bounded
 sub-Riemannian length.
\end{enumerate}
\end{theorem}

The dimension and corank in the theorem are the first possible ones
for a rank-two counterexample.  Indeed, for a bracket-generating
rank-two distribution on a three-dimensional manifold, every
nonconstant singular horizontal curve is contained in the Martinet
surface, which has zero Lebesgue measure
\cite[Section~1 and Appendix~A]{BdSR}. The next case is rank
two and corank two, which is precisely the setting of
\cref{thm:main}.

The distribution constructed in this paper is neither equiregular nor real-analytic.  Therefore, it remains interesting to understand whether the conjecture is true for the corresponding special classes of structures, and in particular for nonholonomic left-invariant distributions on Carnot groups.

Outside a set of endpoints of measure zero, the particular singular
curves produced in the proof have corners and are therefore not length
minimizing. Moreover, the resulting sub-Riemannian structure satisfies
the minimizing Sard conjecture by a theorem of Rifford; see
\cref{sec:minimizing}.

\subsection*{Idea of the construction}

One well-known counterexample to the classical Morse--Sard theorem for
functions on infinite-dimensional spaces is due to Kupka. He constructed
a smooth function on $\ell^2$ having a compact binary cube of critical
points whose critical values fill an interval \cite{Kupka}; see also
\cite[Section~4.3]{LRT}, where we studied this mechanism in greater
detail. Our construction has two ingredients. The first realizes
Kupka's mechanism within the endpoint map of a distribution and
produces arbitrary vertical displacements by means of singular loops.
The second moves the base endpoint without changing the vertical
coordinates.

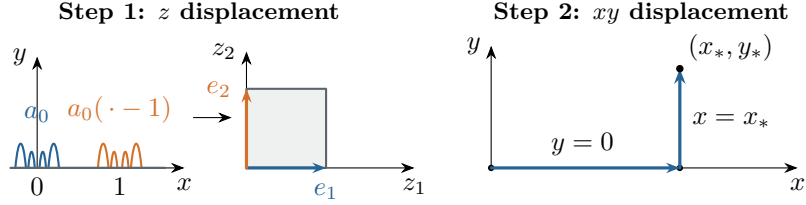
\begin{figure}[t]
\centering
\begin{tikzpicture}[x=.78cm,y=.95cm,font=\small]
  \begin{scope}
    \node[font=\footnotesize\bfseries] at (3.20,2.15) {Step 1: $z$ displacement};

    \draw[-{Stealth[length=1.8mm]}] (0,0) -- (2.95,0) node[below] {$x$};
    \draw[-{Stealth[length=1.8mm]}] (.45,0) -- (.45,1.55) node[left] {$y$};
    \draw[kupkablue,thick]
      (0.08,0) .. controls (0.13,.45) and (0.21,.45) .. (0.26,0)
      (0.30,0) .. controls (0.33,.30) and (0.38,.30) .. (0.41,0)
      (0.49,0) .. controls (0.52,.30) and (0.57,.30) .. (0.60,0)
      (0.64,0) .. controls (0.69,.45) and (0.77,.45) .. (0.82,0);
    \draw[kupkaorange,thick]
      (1.48,0) .. controls (1.53,.45) and (1.61,.45) .. (1.66,0)
      (1.70,0) .. controls (1.73,.30) and (1.78,.30) .. (1.81,0)
      (1.89,0) .. controls (1.92,.30) and (1.97,.30) .. (2.00,0)
      (2.04,0) .. controls (2.09,.45) and (2.17,.45) .. (2.22,0);
    \draw[kupkagray,thick] (0,0) -- (2.62,0) -- (0,0);
    \node[below] at (.45,0) {$0$};
    \node[below] at (1.85,0) {$1$};
    \node[text=kupkablue,above] at (.45,.49) {$a_0$};
    \node[text=kupkaorange,above] at (1.85,.49) {$a_0(\,\cdot-1)$};

    \draw[-{Stealth[length=1.8mm]}] (3.10,0.70) -- (3.75,0.70);

    \draw[-{Stealth[length=1.8mm]}] (4.00,0) -- (6.85,0) node[below] {$z_1$};
    \draw[-{Stealth[length=1.8mm]}] (4.00,0) -- (4.00,1.62) node[left] {$z_2$};
    \filldraw[fill=kupkagray!10,draw=kupkagray,thick]
      (4.00,0) -- (5.35,0) -- (5.35,1.10) -- (4.00,1.10) -- cycle;
    \draw[kupkablue,very thick,-{Stealth[length=1.8mm]}]
      (4.00,0) -- (5.35,0);
    \draw[kupkaorange,very thick,-{Stealth[length=1.8mm]}]
      (4.00,0) -- (4.00,1.10);
    \node[text=kupkablue,below=2pt] at (5.35,0) {$e_1$};
    \node[text=kupkaorange,left=2pt] at (4.00,1.10) {$e_2$};
  \end{scope}

  \begin{scope}[shift={(8.15,0)}]
    \node[font=\footnotesize\bfseries] at (2.55,2.15) {Step 2: $xy$ displacement};

    \draw[-{Stealth[length=1.8mm]}] (0,0) -- (5.20,0) node[below] {$x$};
    \draw[-{Stealth[length=1.8mm]}] (0,0) -- (0,1.70) node[left] {$y$};

    \fill (0,0) circle (1.2pt);
    \fill (3.20,0) circle (1.2pt);
     \draw[kupkablue,very thick,-{Stealth[length=1.8mm]}]
      (0,0) -- (3.20,0);
    \draw[kupkablue,very thick,-{Stealth[length=2mm]}]
      (3.20,0) -- (3.20,1.38);
    \node[above] at (1.55,0.03) {$y=0$};
    \node[right] at (3.24,.66) {$x=x_*$};
    \fill (3.20,1.38) circle (1.4pt)
      node[above right=-1pt] {$(x_*,y_*)$};
  \end{scope}
\end{tikzpicture}
\vspace{-0.6em}
\caption{Schematic illustration of the proof strategy.}
\label{fig:construction-idea}
\end{figure}

We work on $\R^2_{x,y}\times\R^2_z$ with the graph distribution generated
by
\begin{equation}
X=\partial_x-V(x,y)\cdot\partial_z,
\qquad
Y=\partial_y .
\end{equation}
If $(x(t),y(t))$ is the projection of a horizontal curve, its lift
satisfies
\begin{equation}
\dot z(t)=-\dot x(t)V(x(t),y(t)).
\end{equation}
Thus the vertical displacement of a horizontal lift is determined by
$V$. The function $F\coloneq\partial_yV$ 
controls the first-order variation of this displacement when the base
path is varied in the $y$-direction. The criterion proved in
\cref{lem:abnormal} states that if a fixed nonzero multiplier
$\lambda\in\R^2$ satisfies $\lambda\cdot F=0$ along a base path, then
its horizontal lift is singular.

Let $a$ be a smooth compactly supported function, whose specific
construction will be described shortly, and set
\begin{equation}
V(x,y) = \left(\frac{y^3}{3}-\frac{a(x)y^2}{2}\right)w(x), \quad \text{with} \quad  w(x) =(1-x)e_1+xe_2,
\end{equation}
where $e_1,e_2$ are the standard basis vectors of $\R^2_z$. Then
\begin{equation}
F(x,y)=y\bigl(y-a(x)\bigr)w(x).
\end{equation}
Note that this factorization singles out two branches, $y=0$ and $y=a(x)$, on
which $F$ vanishes. Hence every horizontal curve whose $xy$-projection
is contained in their union is singular for \emph{every} nonzero multiplier. 

Fix now an arbitrary target point
\begin{equation}
q_*=(x_*,y_*,z_*)\in\R^2_{x,y}\times\R^2_z .
\end{equation}
We reach $q_*$ in two steps,  illustrated in \cref{fig:construction-idea}. In the first, we construct a horizontal
curve from the origin to $(0,0,z_*)$, so its  $xy$-projection is a loop.
In the second, we move from $(0,0,z_*)$ to $q_*$ without changing the
$z$-coordinate. 

The first step is where the specific construction of $a$ enters. The function
$a$ is assembled from two disjoint systems of binary bumps, one
centered at $x=0$ and the other at $x=1$. A base loop crosses the
supports of these bumps, following either their graphs or the
$x$-axis, and then returns to the origin along the $x$-axis. The
$n$-th bumps in the two systems are normalized so that choosing to
follow their graphs produces, respectively, the vertical increments
$2^{-n}e_1$ and $2^{-n}e_2$. The binary expansions of the two coordinates of $z_*$ determine,
independently for each bump, whether the loop follows its graph or
remains on the $x$-axis. The resulting base path is closed, while its
horizontal lift accumulates the prescribed vertical displacement. Since the base path is contained in the zero set of $F$, the horizontal lift is therefore
singular for every nonzero multiplier and joins the origin to
$(0,0,z_*)$. 

For the second step, starting from $(0,0,z_*)$, we first move along the
$x$-axis until the $x$-coordinate is $x_*$ and then move in the
$y$-direction while keeping $x=x_*$. Neither part changes the
$z$-coordinate. 

Finally, since $F$ vanishes throughout the first step,  the nonzero multiplier given by the second step is also a multiplier for the entire concatenation, which is therefore singular.

\subsection*{Context}
Positive results for the Sard conjecture for general distributions have relied on low-dimensional,
generic, or analytic features.  Zelenko and Zhitomirskii proved a
strong form of the conjecture for generic rank-two distributions on
three-dimensional manifolds \cite{ZZ}.  This was extended to
distributions with a smooth Martinet surface by Belotto da Silva and
Rifford \cite{BdSR}, and to arbitrary analytic rank-two distributions
in dimension three by Belotto da Silva, Figalli, Parusi\'nski, and
Rifford \cite{BdSFPR}.  Belotto da Silva, Parusi\'nski, and Rifford
proved the conjecture for every smooth rank-three distribution in
dimension four and for generic smooth distributions of corank one
\cite{BdSPR}.  They also developed a general description of singular
curves for analytic distributions \cite{BdSPRSubanalytic}, yielding the minimal-rank Sard property under suitable assumptions \cite{BdSPRMinimalRank}.

A parallel line concerns homogeneous structures.  The conjecture is true for Carnot groups of step at most two and for several further
classes \cite{AGL,LDMOPV}.  Subsequent work obtained sharper
codimension estimates in step two \cite{OV}, established the result for
rank-two groups of step four and rank-three groups of step three
\cite{BV}, and treated all step-two and all filiform Carnot groups
\cite{BNV}.  Sard-type results also hold for the restriction of the
endpoint map of an arbitrary Carnot group to sufficiently regular
classes of controls, including piecewise real-analytic and piecewise
entire controls \cite{LRT}.  More recently, the Sard property together
with a sharp rectifiability estimate has been established for rank-two
structures on metabelian Lie groups \cite{LDLNPR}. 

For comparison, the \emph{minimizing} Sard conjecture postulates that the set of
points reached from a fixed point by singular \emph{length-minimizing} curves has
measure zero; see \cite[Conjecture~1]{RiffordTrelat} and also the second
question in \cite[Problem~III]{AgrachevOpen}.  Building on a density result of Rifford and Tr\'elat
\cite{RiffordTrelat}, Agrachev
\cite{AgrachevSmooth} proved that the set of endpoints of singular minimizing curves has
empty interior.
Building on earlier genericity results of Agrachev and Gauthier \cite{AGauthier}, Chitour, Jean, and Tr\'elat proved that generic sub-Riemannian structures of rank at least three admit no nontrivial singular minimizing curves \cite{CJT}.  More recently, Rifford
\cite{RiffordMinSard} proved the minimizing Sard conjecture for
 sub-Riemannian structures associated either with corank-two
distributions or with generic distributions of rank at least two.


\section{Graph distributions and a singularity criterion}

We use coordinates
\begin{equation}
 (x,y,z_1,z_2)\in\R^4=\R^2_{x,y}\times\R^2_z.
\end{equation}
Let $V=(V_1,V_2):\R^2_{x,y}\to\R^2$ be smooth and define
\begin{equation}\label{eq:fields}
 X\coloneq \partial_x-V_1(x,y)\partial_{z_1}
                  -V_2(x,y)\partial_{z_2},
 \qquad
 Y\coloneq \partial_y.
\end{equation}
Set
\begin{equation}
 \Delta\coloneq \Span\{X,Y\},
 \qquad
 F(x,y)\coloneq \partial_yV(x,y)\in\R^2.
\end{equation}
The ordered pair $(X,Y)$ is a global frame of $\Delta$.  Hence every
$W^{1,2}$ horizontal curve parametrized on an interval $[0,L]$ has a unique $L^2$ control relative to this
frame, namely $(u,v)\in L^2([0,L],\R^2)$ such that $\dot\gamma = u X + v Y$. Conversely, since the system
\begin{equation}\label{eq:triangular-system}
 \dot x=u,
 \qquad
 \dot y=v,
 \qquad
 \dot z=-uV(x,y),
\end{equation}
is triangular, every $(u,v)\in L^2([0,L],\R^2)$ determines a unique
horizontal curve $\gamma_{u,v}$ once the initial point is fixed.  We denote by
\begin{equation}
 \End_0^L:L^2([0,L],\R^2)\longrightarrow\R^4,
 \qquad
 \End_0^L(u,v)\coloneq \gamma_{u,v}(L),
\end{equation}
the corresponding control representation of the endpoint map, based at the origin.  For $L=1$ we simply write $\End_0$. Since $\End_0(u,v)=\mathcal E_0(\gamma_{u,v})$, and controls provide a smooth global chart for the $L^2$ manifold $\Omega_0$, a control is a critical point precisely when its associated horizontal curve is singular.

\begin{lemma}[Criterion for singular curves]\label{lem:abnormal}
Let $L>0$ and let $(u,v)\in L^2([0,L],\R^2)$.  Denote by
\begin{equation}
 \gamma=\gamma_{u,v}=(x,y,z):[0,L]\longrightarrow\R^4
\end{equation}
the associated horizontal trajectory starting at the origin.  Suppose that for some 
$\lambda \in\R^2\setminus\{0\}$,
\begin{equation}
 \lambda\cdot F(x(t),y(t))=0
 \qquad\text{for a.e. }t\in[0,L].
\end{equation}
Then $(u,v)$ is a critical point of $\End_0^L$.  Moreover, the
control of the affine reparametrization
$\bar\gamma(t)\coloneq \gamma(Lt)$, $t\in[0,1]$, is critical for $\End_0$.
\end{lemma}

\begin{proof}
By \eqref{eq:triangular-system}, it holds
\begin{equation}
 x(t)=\int_0^t u(s)\dd s,
 \qquad
 y(t)=\int_0^t v(s)\dd s,
\end{equation}
and therefore
\begin{equation}
 \End_0^L(u,v)
 =
 \left(
 x(L),y(L),
 -\int_0^L u(t)V(x(t),y(t))\dd t
 \right).
\end{equation}
Let $(\widehat u,\widehat v)\in L^2([0,L],\R^2)$ be arbitrary, and set
\begin{equation}
 \delta x(t)\coloneq \int_0^t\widehat u(s)\dd s,
 \qquad
 \delta y(t)\coloneq \int_0^t\widehat v(s)\dd s.
\end{equation}
Differentiating the endpoint map in the direction
$(\widehat u,\widehat v)$ gives
\begin{equation}
 D_{(u,v)}\End_0^L(\widehat u,\widehat v)
 =
 \bigl(\delta x(L),\delta y(L),\delta z(L)\bigr),
\end{equation}
where
\begin{equation}
 \delta z(L)
 =
 -\int_0^L
 \left[
 \widehat u V
 +uV_x \delta x
 +uF \delta y
 \right]\dd t.
\end{equation}
Here and below, $V$, $V_x$, and $F$ inside the integrals are evaluated at
$(x(t),y(t))$.

Taking the scalar product with $\lambda$, we obtain
\begin{equation}\label{eq:identityabove}
 \lambda\cdot\delta z(L)
 =
 -\int_0^L
 \left[
 (\lambda\cdot V)\widehat u
 +u(\lambda\cdot V_x)\delta x
 +u(\lambda\cdot F)\delta y
 \right]\dd t.
\end{equation}
Along the reference trajectory,
\begin{equation}
 \frac{\dd}{\dd t}\bigl(\lambda\cdot V(x(t),y(t))\bigr)
 =
 u(\lambda\cdot V_x)+v(\lambda\cdot F)
 \qquad\text{for a.e. }t,
\end{equation}
and
\begin{equation}
 \frac{\dd}{\dd t}\delta x(t)=\widehat u(t)
 \qquad\text{for a.e. }t,
 \qquad
 \delta x(0)=0.
\end{equation}
Hence integration by parts gives
\begin{equation}
 \int_0^L(\lambda\cdot V)\widehat u\dd t
 =
 (\lambda\cdot V(x(L),y(L)))\delta x(L)
 -
 \int_0^L
 \left[
 u(\lambda\cdot V_x)+v(\lambda\cdot F)
 \right]\delta x\dd t.
\end{equation}
Substituting this identity in \eqref{eq:identityabove} yields
\begin{equation}
 \lambda\cdot\delta z(L)
 =
 -(\lambda\cdot V(x(L),y(L)))\delta x(L)
 +
 \int_0^L
 (\lambda\cdot F)
 \bigl(v\delta x-u\delta y\bigr)\dd t.
\end{equation}
By assumption the integral vanishes. Therefore, defining $\eta_L \in \R^4$ by
\begin{equation}
 \eta_L
 \coloneq 
 \bigl(
 \lambda\cdot V(x(L),y(L)), 0, \lambda
 \bigr),
\end{equation}
it holds 
\begin{equation}
 \operatorname{Im}D_{(u,v)}\End_0^L\subset\ker\eta_L.
\end{equation}
Since $\lambda\neq0$, the covector $\eta_L$ is nonzero.  Hence $D_{(u,v)}\End_0^L$ is not surjective.

Finally, define the linear isomorphism
\begin{equation}
 R_L:L^2([0,L],\R^2)\longrightarrow L^2([0,1],\R^2),
 \qquad
 R_L(u,v)(t)\coloneq L\,(u(Lt),v(Lt)).
\end{equation}
The trajectory associated with $R_L(u,v)$ is $t\mapsto\gamma(Lt)$, and
\begin{equation}
 \End_0\circ R_L=\End_0^L.
\end{equation}
Since $R_L$ is an isomorphism, the two endpoint differentials have the
same image.  The reparametrized control is therefore a critical point for
$\End_0$.
\end{proof}

\section{Construction of the distribution}

\subsection{The binary coefficient}
The purpose of this subsection is to construct a compact family of smooth
functions whose graphs independently select binary bumps.  The symmetry
imposed below makes the first moment of each bump vanish.  This will allow
the two translated systems to produce
exactly the vertical directions $e_1$ and $e_2$.

Fix $r=1/4$ and a nonzero nonnegative function
$\psi\in C_c^\infty((0,1))$.  Put
\begin{equation}
 C_\psi\coloneq \int_0^1\psi(s)^3\dd s,
 \qquad
 I_n\coloneq \left(\frac{r}{n+1},\frac{r}{n}\right),
 \qquad
 \ell_n\coloneq |I_n|=\frac{r}{n(n+1)}.
\end{equation}
For $x\in I_n$, define
\begin{equation}
 \rho_n(x)\coloneq
 \kappa_n\psi\!\left(\frac{x-r/(n+1)}{\ell_n}\right),
 \qquad
 \kappa_n^3\coloneq\frac{3\,2^{-n}}{\ell_nC_\psi},
\end{equation}
and extend $\rho_n$ by zero.  Set
\begin{equation}
 a_n(x)\coloneq \rho_n(x)+\rho_n(-x),
 \qquad
 a_0(x)\coloneq\sum_{n\geq1}a_n(x).
\end{equation}
The supports of the $a_n$ are pairwise disjoint, each $a_n$ is even, and
the family $(\operatorname{supp}a_n)_{n\geq1}$ accumulates only at the
origin.

Let $E=\{0,1\}^{\N}$ with the product topology; it is compact.  For
$\varepsilon=(\varepsilon_n)_{n\geq1}\in E$, define
\begin{equation}
 y_\varepsilon(x)\coloneq\sum_{n\geq1}\varepsilon_na_n(x),
 \qquad
 b(\varepsilon)\coloneq\sum_{n\geq1}\varepsilon_n2^{-n}.
\end{equation}
Thus $y_\varepsilon$ independently selects either zero or the full symmetric
bump on each pair of intervals. See \cref{fig:binary-bubbles}.

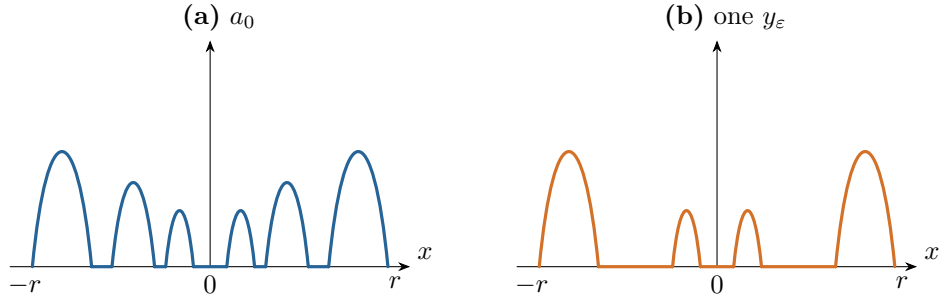
\begin{figure}[t]
\centering
\begin{minipage}[t]{0.47\textwidth}
\centering
\small\textbf{(a)} $a_0$\par\vspace{0.25em}
\begin{tikzpicture}[x=2.45cm,y=2.6cm]
 \draw[-{Stealth[length=1.6mm]}] (-1.08,0)--(1.08,0)
   node[above right=-1pt] {$x$};
 \draw[-{Stealth[length=1.6mm]}] (0,0)--(0,1.15);
 \node[below] at (-1,0) {$-r$};
 \node[below] at (0,0) {$0$};
 \node[below] at (1,0) {$r$};
 \draw[kupkablue,very thick]
  (-.96,0)..controls(-.88,.78)and(-.72,.78)..(-.64,0)
  (-.64,0)--(-.53,0)
  (-.53,0)..controls(-.47,.57)and(-.36,.57)..(-.30,0)
  (-.30,0)--(-.24,0)
  (-.24,0)..controls(-.20,.38)and(-.13,.38)..(-.09,0)
  (-.09,0)--(.09,0)
  (.09,0)..controls(.13,.38)and(.20,.38)..(.24,0)
  (.24,0)--(.30,0)
  (.30,0)..controls(.36,.57)and(.47,.57)..(.53,0)
  (.53,0)--(.64,0)
  (.64,0)..controls(.72,.78)and(.88,.78)..(.96,0);
\end{tikzpicture}
\end{minipage}\hfill
\begin{minipage}[t]{0.47\textwidth}
\centering
\small\textbf{(b)} one $y_\varepsilon$\par\vspace{0.25em}
\begin{tikzpicture}[x=2.45cm,y=2.6cm]
 \draw[-{Stealth[length=1.6mm]}] (-1.08,0)--(1.08,0)
   node[above right=-1pt] {$x$};
 \draw[-{Stealth[length=1.6mm]}] (0,0)--(0,1.15);
 \node[below] at (-1,0) {$-r$};
 \node[below] at (0,0) {$0$};
 \node[below] at (1,0) {$r$};
 \draw[kupkaorange,very thick]
  (-.96,0)..controls(-.88,.78)and(-.72,.78)..(-.64,0)
  (-.64,0)--(-.24,0)
  (-.24,0)..controls(-.20,.38)and(-.13,.38)..(-.09,0)
  (-.09,0)--(.09,0)
  (.09,0)..controls(.13,.38)and(.20,.38)..(.24,0)
  (.24,0)--(.64,0)
  (.64,0)..controls(.72,.78)and(.88,.78)..(.96,0);
\end{tikzpicture}
\end{minipage}
\caption{The symmetric accumulating binary bumps and one selected function
$y_\varepsilon$.  On each pair of intervals, $y_\varepsilon$ either vanishes
or coincides with the corresponding bump $a_n$.}
\label{fig:binary-bubbles}
\end{figure}

\begin{lemma}[The binary gadget]\label{lem:binary}
The following properties hold.
\begin{enumerate}
 \item The function $a_0$ belongs to $C_c^\infty(\R)$, has support in
 $(-r,r)$, and is flat at the origin.  Moreover,
 \begin{equation}\label{eq:digit-moments}
  \frac16\int_\R a_n(x)^3\dd x=2^{-n},
  \qquad
  \int_\R xa_n(x)^3\dd x=0
  \qquad(n\geq1).
 \end{equation}
 \item The functions $y_\varepsilon$ belong to
 $W^{1,2}(\R)\cap C_c^\infty(\R)$, have support in $(-r,r)$, and
 the map $\varepsilon\mapsto y'_\varepsilon$ is continuous from $E$ to
 $L^2(\R)$.
 \item The map $b:E\to[0,1]$ is continuous and surjective.
\end{enumerate}
\end{lemma}

\begin{proof}
The normalization of $\kappa_n$ and the evenness of $a_n$ give
\begin{equation}
 \frac16\int_\R a_n^3\dd x
 =\frac13\kappa_n^3\ell_nC_\psi=2^{-n},
\end{equation}
while the first moment vanishes by parity.  Moreover, for every $k\geq0$,
\begin{equation}
 \|a_n^{(k)}\|_\infty
 \leq C_k2^{-n/3}[n(n+1)]^{k+1/3}\longrightarrow0.
\end{equation}
Thus all derivatives extend continuously by zero at  $x=0$, proving \textup{(i)}.

The same rescaling gives
\begin{equation}
 \|a_n'\|_{L^2}^2
 \leq C2^{-2n/3}[n(n+1)]^{5/3},
 \qquad
 \|a_n\|_{L^2}^2
 \leq C2^{-2n/3}[n(n+1)]^{-1/3}.
\end{equation}
In particular,
\begin{equation}\label{eq:w12-estimate}
 \sup_{n\geq1}n^2\|a_n\|_{W^{1,2}(\R)}^2<\infty,
\end{equation}
and the squared $W^{1,2}$ norms of the $a_n$ are summable.  Smoothness
of every $y_\varepsilon$ follows as in \textup{(i)}.  If $\varepsilon$
and $\tilde\varepsilon$ agree through the $N$-th digit, disjointness of
the supports yields
\begin{equation}
 \|y'_\varepsilon-y'_{\tilde\varepsilon}\|_{L^2(\R)}^2
 \leq\sum_{n>N}\|a_n'\|_{L^2(\R)}^2\longrightarrow0,
\end{equation}
which proves \textup{(ii)}.

Finally, for \textup{(iii)}, the same tail argument gives continuity
of $b$, while surjectivity is the usual binary expansion of points of
$[0,1]$.
\end{proof}

\subsection{The global coefficient \texorpdfstring{$V$}{V}}

The goal of this subsection is to construct the smooth coefficient
\begin{equation}
 V:\R^2_{x,y}\longrightarrow\R^2
\end{equation}
defining the graph distribution.  Define
\begin{equation}
 a(x)\coloneq a_0(x)+a_0(x-1).
\end{equation}
The two summands have disjoint supports, and \cref{lem:binary} shows
that $a\in C_c^\infty(\R)$.  We set
\begin{equation}\label{eq:w}
 w(x)\coloneq(1-x)e_1+xe_2
\end{equation}
and
\begin{equation}
 V(x,y)\coloneq
 \left(\frac{y^3}{3}-\frac{a(x)y^2}{2}\right)w(x).
\end{equation}
Then
\begin{equation}\label{eq:F-global}
 F(x,y)=\partial_yV(x,y)
 =y\bigl(y-a(x)\bigr)w(x).
\end{equation}
Thus $V$ is the coefficient defining the graph distribution, while $F$
will control both the singularity criterion and the vertical brackets.

With this choice of $V$, let $X,Y$ be the vector fields in
\eqref{eq:fields} and set $\Delta\coloneq\Span\{X,Y\}$.  The identities
used below are
\begin{equation}\label{eq:root-identities}
 \begin{gathered}
  F(x,0)=F(x,a(x))=0,
  \qquad
  V(x,0)=0,\\
  V(x,a(x))=-\frac{a(x)^3}{6}w(x).
 \end{gathered}
\end{equation}

\begin{proposition}\label{prop:bracket-generating}
The distribution $\Delta$ is bracket-generating at every point of
$\R^4$, and its maximal step is five.  It is not equiregular.
\end{proposition}

\begin{proof}
Since $F=V_y$ and the coefficients are
independent of $z$,
\begin{equation}
 \begin{aligned}
  {}[Y,X]&=-F\cdot\partial_z,\\
  [Y,[Y,[Y,X]]]&=-F_{yy}\cdot\partial_z
  =-2w(x)\cdot\partial_z,\\
  [X,[Y,[Y,[Y,X]]]]&=-2(e_2-e_1)\cdot\partial_z,
 \end{aligned}
\end{equation}
where, we recall, $w(x)= (1-x)e_1+xe_2$. Since
\begin{equation}
 \det\bigl(w(x),e_2-e_1\bigr)=1,
\end{equation}
the vectors $w(x)$ and $e_2-e_1$ are independent for every $x$.
Consequently, brackets of length at most five
span $T\R^4$ at every point.

To see that the bound is sharp, choose $x_0$ outside the support of $a$.
In a neighborhood of $(x_0,0)$ one has
\begin{equation}
 V(x,y)=\frac{y^3}{3}w(x),
 \qquad
 F(x,y)=y^2w(x).
\end{equation}
Vertical vector fields commute, while bracketing
$G(x,y)\cdot\partial_z$ with $X$ or $Y$ differentiates $G$ in $x$ or
$y$, up to sign.  Hence vertical brackets of length $j\geq2$ come from
derivatives of $F=y^2w$ of order $j-2$.  At $(x_0,0)$, those of order at
most one vanish, those of order two span only $w(x_0)$, and the second
vertical direction first appears at order three through
$\partial_x\partial_y^2F(x_0,0)=2(e_2-e_1)$.  Thus the maximal step is five.

Finally, wherever $p(x,y)\coloneq y(y-a(x))$ is nonzero, the bracket
\begin{equation}
 [Y,X]=-p w\cdot\partial_z
\end{equation}
is nonzero, so $\Delta+[\Delta,\Delta]$ has dimension three.  At
$(x_0,0)$ it has dimension two, proving that $\Delta$ is not equiregular.
\end{proof}

\section{Global singular reachability}

\subsection{Closed curves producing vertical displacements}
Recall that $r=1/4$.  For $\varepsilon,\delta\in E=\{0,1\}^{\N}$,
define two closed base curves at $(0,0)$, each parametrized on $[0,3]$,
by
\begin{equation}
 \sigma^0_\varepsilon(t)\coloneq
 \begin{cases}
  (-rt,0),&0\leq t\leq1,\\
  \bigl(-r+2r(t-1),y_\varepsilon(-r+2r(t-1))\bigr),
       &1\leq t\leq2,\\
  (r(3-t),0),&2\leq t\leq3,
 \end{cases}
\end{equation}
and
\begin{equation}
 \sigma^1_\delta(t)\coloneq
 \begin{cases}
  ((1-r)t,0),&0\leq t\leq1,\\
  \bigl(1-r+2r(t-1),y_\delta(-r+2r(t-1))\bigr),
       &1\leq t\leq2,\\
  ((1+r)(3-t),0),&2\leq t\leq3.
 \end{cases}
\end{equation}

The two families of closed curves are illustrated in
\cref{fig:vertical-loops}.

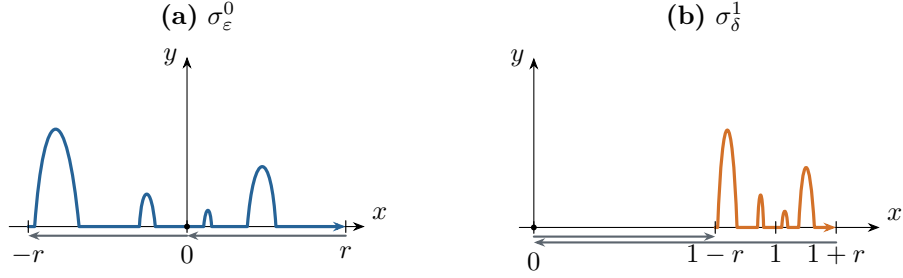
\begin{figure}[t]
\centering
\begin{minipage}[t]{0.47\textwidth}
\centering
\small\textbf{(a)} $\sigma^0_\varepsilon$\par\vspace{0.25em}
\begin{tikzpicture}[x=2.1cm,y=2.2cm]
 \draw[-{Stealth[length=1.6mm]}] (-1.12,0)--(1.12,0)
   node[above right=-1pt] {$x$};
 \draw[-{Stealth[length=1.6mm]}] (0,0)--(0,1.02)
   node[left] {$y$};
 \draw[kupkagray,thick,-{Stealth[length=1.5mm]}]
   (0,-.055)--(-1,-.055);
 \draw[kupkagray,thick,-{Stealth[length=1.5mm]}]
   (1,-.055)--(0,-.055);
 \draw[kupkablue,very thick,-{Stealth[length=1.7mm]}]
  (-1,0)--(-.96,0)..controls(-.90,.78)and(-.76,.78)..(-.68,0)
  --(-.30,0)..controls(-.285,.26)and(-.225,.26)..(-.20,0)
  --(.105,0)..controls(.117,.13)and(.147,.13)..(.155,0)
  --(.38,0)..controls(.42,.48)and(.53,.48)..(.56,0)--(1,0);
 \foreach \t in {-1,0,1}
  \draw[black,line width=.45pt] (\t,-.035)--(\t,.035);
 \fill (0,0) circle (1.1pt);
 \node[below=4pt] at (-1,0) {$-r$};
 \node[below=4pt] at (0,0) {$0$};
 \node[below=4pt] at (1,0) {$r$};
\end{tikzpicture}
\end{minipage}\hfill
\begin{minipage}[t]{0.47\textwidth}
\centering
\small\textbf{(b)} $\sigma^1_\delta$\par\vspace{0.25em}
\begin{tikzpicture}[x=3.2cm,y=2.2cm]
 \draw[-{Stealth[length=1.6mm]}] (-.06,0)--(1.43,0)
   node[above right=-1pt] {$x$};
 \draw[-{Stealth[length=1.6mm]}] (0,0)--(0,1.02)
   node[left] {$y$};
 \draw[kupkagray,thick,-{Stealth[length=1.5mm]}]
   (0,-.055)--(.75,-.055);
 \draw[kupkagray,thick,-{Stealth[length=1.5mm]}]
   (1.25,-.09)--(0,-.09);
 \draw[kupkaorange,very thick,-{Stealth[length=1.7mm]}]
  (.75,0)--(.76,0)..controls(.78,.78)and(.82,.78)..(.84,0)
  --(.925,0)..controls(.93,.26)and(.945,.26)..(.95,0)
  --(1.025,0)..controls(1.03,.13)and(1.045,.13)..(1.05,0)
  --(1.095,0)..controls(1.11,.48)and(1.14,.48)..(1.16,0)--(1.25,0);
 \foreach \t in {0,.75,1,1.25}
  \draw[black,line width=.45pt] (\t,-.035)--(\t,.035);
 \fill (0,0) circle (1.1pt);
 \node[below=6pt] at (0,0) {$0$};
 \node[below=3pt] at (.75,0) {$1-r$};
 \node[below=3pt] at (1,0) {$1$};
 \node[below=3pt] at (1.25,0) {$1+r$};
\end{tikzpicture}
\end{minipage}
\caption{Representative curves from the two closed families.  The colored
parts follow the selected binary graphs from left to right; the gray arrows
indicate the portions on the $x$-axis that close each curve at the origin.}
\label{fig:vertical-loops}
\end{figure}

\begin{proposition}[Vertical displacements from closed curves]\label{prop:vertical-curves}
The horizontal lifts of $\sigma^0_\varepsilon$ and
$\sigma^1_\delta$ starting at the origin end respectively at
\begin{equation}
 \bigl(0,0,b(\varepsilon)e_1\bigr),
 \qquad
 \bigl(0,0,b(\delta)e_2\bigr).
\end{equation}
Along both curves one has $F(x(t),y(t))=0$ almost everywhere.
Moreover, for every $z\in\R^2$ there is a finite concatenation of these
closed curves and their time reversals whose horizontal lift joins
$(0,0,0)$ to $(0,0,z)$ and still satisfies $F=0$ almost everywhere.
\end{proposition}

\begin{proof}
The axis portions contribute no vertical displacement.  Along the
middle portion of $\sigma^0_\varepsilon$, the graph height is pointwise
either zero or $a(x)$.  Thus \eqref{eq:root-identities} and
\eqref{eq:digit-moments}, together with
$w(x)=e_1+x(e_2-e_1)$, give
\begin{equation}\label{eq:first-direction-moment}
 \frac16\int_\R a_n(x)^3w(x)\dd x=2^{-n}e_1.
\end{equation}
Consequently,
\begin{equation}
 -\int_{-r}^rV(x,y_\varepsilon(x))\dd x
 =\sum_{n\geq1}\varepsilon_n2^{-n}e_1
 =b(\varepsilon)e_1.
\end{equation}
For $\sigma^1_\delta$, write $x=1+s$.  Since
$w(1+s)=e_2+s(e_2-e_1)$, the same moment identities give
\begin{equation}\label{eq:second-direction-moment}
 \frac16\int_\R a_n(s)^3w(1+s)\dd s=2^{-n}e_2.
\end{equation}
Thus the displacement of this curve is $b(\delta)e_2$.  This proves the two
endpoint formulas.  Since the graph height is zero or $a(x)$,
\eqref{eq:root-identities} also gives $F=0$ along both curves.

Since $e_1,e_2$ are a basis, it is enough to realize an arbitrary real
multiple of either vector.  For either direction and $\alpha\in\R$, write
$|\alpha|=N+\theta$ with $N\in\N\cup\{0\}$ and $\theta\in[0,1)$.
By \cref{lem:binary}\textup{(iii)}, the surjectivity of $b$ gives a
closed curve with displacement $\theta$ times the chosen direction,
while $N$ copies of a curve with unit displacement produce the integer
part.  If $\alpha<0$, reverse these curves in time.  Indeed, replacing a
control $\mathsf{u}(t)$ by $-\mathsf{u}(L-t)$ changes $z(L)$ to its negative and preserves $F=0$.
Translation invariance in the $z$ variables allows these curves to be
concatenated at any vertical point.
\end{proof}
\begin{remark}[Kupka's map inside the endpoint map]
\label{sec:kupka-map}

We conclude by making precise how Kupka's counterexample (in the form introduced in \cite{yomdincomte}) occurs inside
the endpoint map of the distribution constructed above.  Set
\begin{equation}
 \phi(s)\coloneq3s^2-2s^3.
\end{equation}
Note that $\phi(0)=\phi(1)-1=0$ and $\phi'(0)=\phi'(1)=0$. Kupka's map  is the following infinite dimensional cubic polynomial:
\begin{equation}
 P:\ell^2\longrightarrow\R,
 \qquad
 P(\xi)\coloneq\sum_{n\geq1}2^{-n}\phi(n\xi_n).
\end{equation}
The critical set of $P$ contains the binary cube
\begin{equation}
 \mathcal Q\coloneq
 \left\{\xi\in\ell^2:\xi_n\in\left\{0,\frac1n\right\}
 \text{ for every }n\right\}.
\end{equation}
The image of this cube under $P$ is the whole segment $[0,1]$.

We now identify this map directly in the endpoint construction. To this end, we define the  continuous injective linear map
\begin{equation}
 T:\ell^2\longrightarrow W^{1,2}(\R),
 \qquad
 T(\xi)\coloneq\sum_{n\geq1}n\xi_na_n.\end{equation}
For
$\xi,\eta\in\ell^2$, put
\begin{equation}
 h_{\xi,\eta}(x)\coloneq T(\xi)(x)+T(\eta)(x-1).
\end{equation}
Consider the closed base curve that moves from $(0,0)$ to $(-r,0)$ along
the axis, follows the graph of $h_{\xi,\eta}$ from $x=-r$ to $x=1+r$,
and returns to the origin along the axis.  Using affine parametrizations
on three unit time intervals, the endpoint of its horizontal lift is therefore $ 
 \bigl(0,0,P(\xi),P(\eta)\bigr).$
\end{remark}
\subsection{Singular curves in the base plane}

The affine $x$-dependence of the direction $w(x)$ makes movement in
the base plane immediate.

\begin{lemma}[Singular curves in  the base plane]\label{lem:base-access}
For every $(x_*,y_*)\in\R_{x,y}^2$, define the broken line
\begin{equation}\label{eq:base-path}
 \omega_{x_*,y_*}(t)\coloneq
 \begin{cases}
  (tx_*,0),&0\leq t\leq1,\\
  (x_*,(t-1)y_*),&1\leq t\leq2.
 \end{cases}
\end{equation}
Its horizontal lift has zero vertical displacement, and the path admits
the nonzero multiplier $\lambda_*=(x_*,x_*-1)$; namely,
\begin{equation}
 \lambda_*\cdot F(\omega_{x_*,y_*}(t))=0
 \qquad\text{for a.e. }t.
\end{equation}
\end{lemma}

\begin{proof}
The horizontal segment lies on $y=0$, where $F=V=0$.  On the vertical
segment one has $x=x_*$ and $\dot x=0$, so it contributes no vertical
displacement.  Moreover, \eqref{eq:F-global} gives
\begin{equation}
 \lambda_*\cdot F(x_*,y)
 =y\bigl(y-a(x_*)\bigr)
 \bigl(x_*(1-x_*)+(x_*-1)x_*\bigr)=0.\qedhere
\end{equation}
\end{proof}

\begin{proposition}[Every point is an abnormal endpoint]
\label{prop:global-abnormal}
For every $q_*\in\R^4$ there exists a singular control
$\mathsf{u}\in L^2([0,1],\R^2)$ such that $\End_0(\mathsf{u})=q_*$.  
\end{proposition}

\begin{proof}
Write $q_*=(x_*,y_*,z_*)$.  \cref{prop:vertical-curves} gives
a finite concatenation of closed curves from $(0,0,0)$ to $(0,0,z_*)$
along which $F=0$.  Concatenate it with the horizontal lift of
$\omega_{x_*,y_*}$ from \eqref{eq:base-path}.  The vector fields are invariant
under translations in $z$, and \cref{lem:base-access} shows that
the second part has zero vertical displacement.  The endpoint is
therefore $q_*$.

The same multiplier $\lambda_*=(x_*,x_*-1)$ works along the entire
concatenation: on the closed-curve part $F=0$, and on the broken line
the conclusion follows from \cref{lem:base-access}.  The
concatenation has finitely many $W^{1,2}$ pieces and hence an $L^2$
control on a finite interval.  After affine reparametrization to
$[0,1]$, \cref{lem:abnormal} proves that this control is a critical
point of $\End_0$.
\end{proof}

\section{A compact singular family with open endpoint image}

We now use the two binary loops and restrict the endpoint in the base plane
to a bounded region.  Set
\begin{equation}
 U\coloneq(1,2)^2,
 \qquad
 \overline U=[1,2]^2.
\end{equation}
For $(x,y)\in\overline U$, let $\omega_{x,y}$ be the broken line defined
in \eqref{eq:base-path}.
By \cref{lem:base-access}, its horizontal lift has zero vertical
displacement and admits the multiplier $(x,x-1)$.

For $\varepsilon,\delta\in E = \{0,1\}^{\mathbb{N}}$ and
$(x,y)\in\overline U$, concatenate $\sigma^0_\varepsilon$,
$\sigma^1_\delta$, and $\omega_{x,y}$, using the fixed parametrizations
above.  The resulting base path is defined on $[0,8]$.  Let its horizontal
lift start at the origin, reparametrize it affinely to $[0,1]$, and denote
the resulting control by $\mathsf{u}_{\varepsilon,\delta,x,y}$.
\begin{proposition}[A compact family of singular controls]\label{prop:compact-family}

The set
\begin{equation}
 \mathcal C\coloneq \left\{
  \mathsf{u}_{\varepsilon,\delta,x,y}:
  \varepsilon,\delta\in E,\ (x,y)\in\overline U
 \right\}
 \subset L^2([0,1],\R^2)
\end{equation}
is compact and consists entirely of singular controls.  Moreover,
\begin{equation}\label{eq:local-endpoint}
 \End_0(\mathsf{u}_{\varepsilon,\delta,x,y})
 =\bigl(x,y,b(\varepsilon),b(\delta)\bigr).
\end{equation}
Consequently, $\End_0(\mathcal C)$ contains the nonempty open set
\begin{equation}
 \mathcal O\coloneq U\times(0,1)^2.
\end{equation}
\end{proposition}

\begin{proof}
\cref{prop:vertical-curves} and \cref{lem:base-access}
give \eqref{eq:local-endpoint}.  Along the closed curves one has $F=0$,
and the multiplier $(x,x-1)$ works along $\omega_{x,y}$.  Hence it works
along the whole concatenation, and \cref{lem:abnormal} proves
singularity.

For compactness, all pieces use fixed parameter intervals.  By
\cref{lem:binary}\textup{(ii)}, the controls of the binary portions
depend continuously in $L^2$ on $\varepsilon$ and $\delta$.  The two
controls of the final broken line are $(x,0)$ and $(0,y)$, so they
depend continuously on $(x,y)$.  Therefore
\begin{equation}
 E\times E\times\overline U
 \longrightarrow L^2([0,1],\R^2),\qquad
 (\varepsilon,\delta,x,y)\longmapsto
 \mathsf{u}_{\varepsilon,\delta,x,y},
\end{equation}
is continuous.  Since the parameter space is compact, so is $\mathcal C$.

The surjectivity of $b$ from \cref{lem:binary}\textup{(iii)} and
\eqref{eq:local-endpoint} give
$\mathcal O\subset\End_0(\mathcal C)$.
\end{proof}

\begin{corollary}[A compact family in the horizontal path space]
\label{cor:compact-horizontal-family}
Let $\mathcal C$ be the compact set in
\cref{prop:compact-family}, and define
\begin{equation}\label{eq:compact-horizontal-family}
 \mathcal K\coloneq \{\gamma_\mathsf{u}:\mathsf{u}\in\mathcal C\}
 \subset\Omega_0.
\end{equation}
Then $\mathcal K$ is compact, consists of singular horizontal
curves, and satisfies
\begin{equation}\label{eq:K-endpoint-image}
 \mathcal E_0(\mathcal K)=\End_0(\mathcal C)\supset\mathcal O.
\end{equation}
Moreover, for any sub-Riemannian metric $g$ on $\Delta$, and denoting by
$\ell_g(\gamma)$ the length of a horizontal curve, it holds
\begin{equation}
 \sup_{\gamma\in\mathcal K}\ell_g(\gamma)<\infty.
\end{equation}
\end{corollary}

\begin{proof}
By the global control parametrization associated with $(X,Y)$, the map
$\mathsf{u}\mapsto\gamma_\mathsf{u}$ is continuous.  Hence
\eqref{eq:compact-horizontal-family} and the compactness of $\mathcal C$
show that $\mathcal K$ is compact.  Its elements are singular, and
\eqref{eq:K-endpoint-image} follows from
\cref{prop:compact-family}.

Finally, compactness of $\mathcal C$ implies boundedness in $L^2$.  Let
$g_0$ be the sub-Riemannian metric making $X,Y$ orthonormal.  Then
\begin{equation}
 \sup_{\gamma\in\mathcal K}\ell_{g_0}(\gamma)
 =\sup_{\mathsf{u}\in\mathcal C}\int_0^1|\mathsf{u}(t)|\dd t
 \leq\sup_{\mathsf{u}\in\mathcal C}\|\mathsf{u}\|_{L^2}<\infty.
\end{equation}
Since the images of the curves $\gamma_\mathsf{u}$, with $\mathsf{u}\in\mathcal C$, are
contained in a common compact set, $g$
and $g_0$ are equivalent there, and the statement follows.
\end{proof}

\begin{proof}[Proof of \cref{thm:main}]
Take $q_0=0$.  \cref{prop:bracket-generating} provides a smooth
rank-two distribution of step five.
\cref{prop:global-abnormal} proves assertion~\textup{(i)}, and
\cref{cor:compact-horizontal-family} proves assertion~\textup{(ii)}.
\end{proof}

\section{The minimizing Sard property}
\label{sec:minimizing}

The counterexample concerns all singular horizontal curves.  In contrast,
the same structure satisfies the minimizing Sard conjecture.

\begin{proposition}
\label{prop:minimizing-sard}
Let $\Delta$ be the sub-Riemannian distribution on $\R^4$ built above, and fix a sub-Riemannian metric $g$ on $\Delta$. For every $q_0\in\R^4$, the set of endpoints of singular length-minimizing horizontal curves issuing from $q_0$ has Lebesgue measure zero.
\end{proposition}

\begin{proof}
Let $d_g$ denote the sub-Riemannian distance associated with $g$. We can assume without loss of generality that $d_g$ is complete; otherwise a standard countable exhaustion argument implies the result as well. The distribution $\Delta$ has rank two on the four-dimensional manifold
$\R^4$, hence corank two. Corank two distributions are pre-medium fat almost everywhere in the sense of \cite{RiffordMinSard}, and the statement thus follows from \cite[Corollary~1.3]{RiffordMinSard}.
\end{proof}

We next record a direct geometric feature of our construction.  
\begin{proposition}
For every
$q_*=(x_*, y_*, z_*)\in\R^4$ with $x_*y_*\neq 0$, each of the curves produced by the
construction in the proof of \cref{prop:global-abnormal} and
ending at $q_*$ contains a corner, and hence is not length minimizing for
any smooth sub-Riemannian metric on $\Delta$.
\end{proposition}

\begin{proof}
 Fix $q_*=(x_*,y_*,z_*)$ with $x_*y_*\neq 0$.  The last
part of every curve constructed in the proof of
\cref{prop:global-abnormal} is the horizontal lift of
$\omega_{x_*,y_*}$ defined in \eqref{eq:base-path}.
Since $x_*y_*\neq0$, both segments are nonconstant.  At their junction,
the two one-sided horizontal directions are spanned by $X$ and $Y$,
respectively, and are therefore linearly independent.  Thus the lifted
curve has a corner.  The preliminary closed curves used to adjust the
$z$-coordinate do not alter this junction.  By the no-corner theorem of
Hakavuori and Le Donne \cite{HLD}, such curves cannot be length minimizing
for any smooth sub-Riemannian metric on $\Delta$.
\end{proof}

The two conclusions are complementary.  \cref{prop:minimizing-sard}
shows that, for every fixed base point, the endpoints of all singular
minimizing curves form a null set.  The corner argument
shows directly that the particular curves used to prove
$\Abn_\Delta(0)=\R^4$ are non-minimizing for almost every endpoint.

\subsection*{Acknowledgements}
This project has received funding from the European Research Council
(ERC) under the European Union's Horizon 2020 research and innovation
programme (grant agreement GEOSUB, No.~945655). The authors
acknowledge the support of INdAM and of the IRP project GEOSUBMAN by
INSMI-CNRS. D.T. has also
received funding from the
European Union -- NextGenerationEU and the University of Padua under the 2023
STARS@UNIPD  Starting Grant Project \textit{New Directions in Fractional
Calculus -- NewFrac} (grant agreement No.\ CUP\_C95\-F21\-009\-990\-001), and from the INdAM--GNAMPA Project 2026 \textit{Metodi
non locali classici e distribuzionali per problemi variazionali} (grant agreement No. \ CUP\_E53\-C25\-002\-010\-001).

\subsection*{AI-assisted tools disclosure.}
In accordance with the Leiden Declaration on AI
and Mathematics \cite{LeidenDeclaration}, we disclose the role of
artificial intelligence in the preparation of this work.  The
conceptual idea underlying the paper -- namely, to realize, or
``endpointify'', Kupka's counterexample within a sub-Riemannian endpoint
map -- was conceived by the authors.  It grew out of our earlier work on
Sard properties for polynomial maps in infinite dimension \cite{LRT}
and, in particular, our study there of Kupka's classical construction.

The mathematical implementation of this idea was AI-generated using
OpenAI's ChatGPT 5.6 Sol.  The authors take full and
exclusive responsibility for the correctness of the results and the
contents of the paper.

\bibliographystyle{abbrv}
\bibliography{no_sard_references}

\end{document}